\documentclass[12pt]{amsart}
\usepackage[T1]{fontenc}
\usepackage{amsfonts}
\usepackage{mathrsfs}
\usepackage{amscd,amsmath,amssymb,amsfonts}

\usepackage[a4paper,text={140mm,240mm},centering,headsep=5mm,footskip=10mm]{geometry}

\usepackage[all]{xy}
\theoremstyle{plain}
\newtheorem{thm}{Theorem}
\newtheorem{lem}[thm]{Lemma}
\newtheorem{cor}[thm]{Corollary}
\newtheorem{prop}[thm]{Proposition}
\newtheorem{conj}[thm]{Conjecture}
\theoremstyle{definition}
\newtheorem{defn}[thm]{Definition}

\newtheorem{question}[thm]{Question}
\newtheorem{rmk}[thm]{Remark}

\newtheorem{claim}[thm]{Claim}

\numberwithin{thm}{section} \numberwithin{equation}{section}

\newcommand{\eq}[2]{\begin{equation}\label{#1}#2 \end{equation}}
\newcommand{\ml}[2]{\begin{multline}\label{#1}#2 \end{multline}}
\newcommand{\ga}[2]{\begin{gather}\label{#1}#2 \end{gather}}

\newcommand{\inj}{\hookrightarrow}
\newcommand{\codim}{{\rm codim}}

\newcommand{\Pic}{{\rm Pic}}

\newcommand{\im}{{\rm im}}
\newcommand{\Spec}{{\rm Spec \,}}

\newcommand{\sC}{{\mathcal C}}

\newcommand{\sF}{{\mathcal F}}

\newcommand{\sH}{{\mathcal H}}

\newcommand{\sK}{{\mathcal K}}
\newcommand{\sKM}{{\mathcal K}^{\rm M}}

\newcommand{\sO}{{\mathcal O}}

\newcommand{\sR}{{\mathcal R}}
\newcommand{\sS}{{\mathcal S}}

\newcommand{\sX}{{\mathcal X}}

\newcommand{\C}{{\rm  C}}    
\newcommand{\CC}{{\mathbb C}}

\renewcommand{\P}{{\mathbb P}}
\newcommand{\Q}{{\mathbb Q}}

\newcommand{\Z}{{\mathbb Z}}

\newcommand{\dR}{{\rm dR}}

\newcommand{\CH}{{\rm CH}}

\newcommand{\ch}{{\rm ch}}

\newcommand{\rKM}{{K}^{\rm M}}

\newcommand{\cl}{{\rm cl}}
\newcommand{\rHC}{{\rm HC}}

\begin{document}

\title[Deformation of cycle classes]{Deformation of algebraic cycle classes in
  characteristic zero}
\author{Spencer Bloch}
\author{H{\'e}l{\`e}ne Esnault}
\author{Moritz Kerz}
\address{5765 S. Blackstone Ave., Chicago, IL 60637,
USA}
\address{ FU Berlin, Mathematik, Arnimallee 3, 14195 Berlin, Germany}
\address{Fakult\"at f\"ur Mathematik, Universit\"at Regensburg, 93040 Regensburg, Germany}
\email{spencer\_bloch@yahoo.com }
\email{esnault@math.fu-berlin.de}
\email{moritz.kerz@mathematik.uni-regensburg.de}
\thanks{The second author is supported by  the  Einstein Foundation,
the ERC Advanced Grant 226257,  the third author by the DFG
Emmy-Noether Nachwuchsgruppe ``Arithmetik \"uber endlich erzeugten K\"orpern''
.}
\begin{abstract} 
We study a formal deformation problem for rational algebraic cycle classes motivated by
Grothendieck's variational Hodge conjecture. We argue that there is a close connection
between the existence of a Chow-K\"unneth decomposition and the existence of expected
deformations of cycles. This observation applies in particular to abelian schemes.
\end{abstract}

\maketitle
\begin{quote}

\end{quote}
\section{Introduction}

\noindent
Let $K$ be a field of characteristic $0$.
Let $Z/K$ be a smooth projective variety of dimension $d$ and let $i$ be an integer. The following Chow-K\"unneth property is 
expected to hold in general as part of conjectures of Grothendieck, Beilinson and Murre, see \cite[Sec.~5]{Ja}.

\medskip

${\bf (CK)}_Z^i \quad$: There is an idempotent correspondence $\pi^i\in \CH^d(Z\times_K Z)_\Q $
such that on $ H^*_\dR(Z/K)$ the correspondence $\pi^i$ acts
as the projection to $H^i_\dR(Z/K)$.

\begin{rmk}\label{int.rmkab}
An important and well-understood example is the case of an abelian variety $A/K$. In this case we know $ {\bf
  (CK)}_A^i$ for all integers $i$, see for instance~\cite{DM}.
\end{rmk}

\medskip

Set $S=\Spec k[[t]]$, where $k$ is a field of characteristic $0$.
Let $X/S$ be a smooth projective scheme. Let $\eta$ be the generic point of $S$,
$K=k(\eta)$. Denote by $X_n$ the scheme $X\times_S S_n$ with $S_n=\Spec k[t]/(t^n)$.

There is a Chern character ring homomorphism to de Rham cohomology 
\begin{equation}
\ch: K_0(X_1) \to H^*_\dR(X_1/k).
\end{equation}
It is an important problem to determine the image of the Chern character. In this note we
consider an `infinitesimal' analogue of this problem along the thickening $X_1 \hookrightarrow
X_n$.

By $\nabla : H^*_\dR(X/S) \to  H^*_\dR(X/S)$ we denote derivation along the parameter $t$
with respect to the Gauss-Manin connection and by
 $ H^*_\dR(X/S)^\nabla$ the kernel of~$\nabla$.
Solving a formal differential equation we see that the canonical map
\begin{equation}\label{int.solv}
\Phi: H^*_\dR(X/S)^\nabla \xrightarrow{\sim} H^*_\dR(X_1/k)
\end{equation}
is an isomorphism \cite[Prop.~8.9]{Katz}. 

We denote by $F^r H^i_\dR \subset H^i_\dR  $ the Hodge filtration on de Rham cohomology (of
smooth projective schemes). Our main theorem says:

\begin{thm}\label{main.thm}
Assume that for the scheme $X/S$ as above the property ${\bf (CK)}_{X_\eta}^i$ holds for all
even $i\in \Z$. Then for $\xi_1\in K_0(X_1)_\Q$ the following are equivalent:
\begin{itemize}
\item[(i)] $\Phi^{-1} \circ \ch(\xi_1) \in  \bigoplus_i   H^{2i}_{\dR}(X/S)^\nabla \cap F^i H^{2i}_\dR(X/S)$,
\item[(ii)] there is an element $\hat \xi \in (\varprojlim_n K_0(X_n) )\otimes \Q $ such
  that $$\ch(\hat \xi |_{X_1}) = \ch (\xi_1) \in H_\dR^*(X_1/k).$$
\end{itemize}
\end{thm}

A preliminary version  of Theorem~\ref{main.thm} for cohomological Chow groups is shown in Section~\ref{Chow}, see
Theorem~\ref{main.thmch}. The central new ingredient of our proof is to study a
ring of correspondences for the non-reduced scheme $X_n$.  Property ${\bf
  (CK)}_{X_\eta}^i$ will guarantee that there are enough such correspondences in order to
kill the influence of absolute differential forms of $k$ on the deformation behavior. In
fact for $\Omega^1_{k/\Q}=0$ the whole deformation problem is much easier, see Remark~\ref{intro.rmkalg}.
The proof of Theorem~\ref{main.thm} is  completed via a
Chern character isomorphism relying on Zariski descent for algebraic $K$-theory, see Section~\ref{motcom}.

\begin{rmk}\label{intro.rmkalg}
In case $k$ is algebraic over $\Q$ we also deduce without assuming ${\bf (CK)}_{X_\eta}^i$
that conditions (i) and (ii) are equivalent and they are also equivalent to:
\begin{itemize}
\item[(ii')] there is an element $\hat \xi \in (\varprojlim_n K_0(X_n) )\otimes \Q $ such
  that $\hat \xi |_{X_1} = \xi_1$.
\end{itemize}
See related work  \cite{GG}, \cite{PaRa}, \cite{Mor}. However, our methods do not show
that for general fields $k$ condition (ii') is equivalent to (ii) and we do not see a good
reason to expect this.
\end{rmk}

Theorem~\ref{main.thm} is motivated by a conjecture of Grothendieck \cite[p.~103]{GdR}, which is today
called the variational Hodge conjecture. See Appendix~\ref{groth.conj} for his
original global formulation. In this appendix it is shown that the latter is equivalent to the
following ``infinitesimal'' conjecture.

\begin{conj}[Infinitesimal Hodge]\label{int.conj}
Statement {\rm (i)}  of Theorem \ref{main.thm}  is equivalent
to
\begin{itemize}
\item[(iii)] there is an element $ \xi \in K_0(X)_\Q $ such
  that $$\ch( \xi |_{X_1}) = \ch (\xi_1) \in H_\dR^*(X_1/k).$$
\end{itemize}
\end{conj} 

One shows directly that (iii) $\Rightarrow$ (ii) $\Rightarrow$ (i) without assuming ${\bf (CK)}_{X_\eta}^i$
for $i$ even.  Conjecture~\ref{int.conj} is particularly interesting for abelian schemes. Indeed it is known \cite{Ab},\cite[Sec.~6]{An} that
\[
\text{Conj. \ref{int.conj} for abelian schemes } X/S  \Longrightarrow \text{ Hodge
  conj. for abelian varieties.}
\]
 
As ${\bf (CK)}_{X_\eta}^i$ is known for abelian varieties, see Remark~\ref{int.rmkab},
one can speculate about what is needed to deduce Conjecture~\ref{int.conj} for abelian schemes from
Theorem~\ref{main.thm}. In order to accomplish this one would have to
solve an algebraization problem, namely one has to consider the question how far the map
\eq{intr.algmap}{ K_0(X) \to \varprojlim_n K_0(X_n) }
is from being surjective (after tensoring with $\Q$). Recall that for line bundles the corresponding map
\[
\Pic(X) \to  \varprojlim_n \Pic(X_n)
\]
is an isomorphism by formal existence~\cite[III.5]{EGA3}.

 By considering a trivial deformation of an abelian surface we show in
Appendix~\ref{counterex}
that the map \eqref{intr.algmap} cannot be surjective in general.
For abelian schemes the counterexample leaves the following algebraization question open.

\begin{question}\label{main.ques}
Let $X/S$ be an abelian scheme and  $\ell>1$ an integer. Is the map
\[K_0(X)^{\Psi^{\ell^2}-[\ell]^*}  \to  \varprojlim_n  K_0(X_n)^{\Psi^{\ell^2}-[\ell]^*}    \]
surjective after tensoring with $\Q$?
\end{question}

Here $\psi^\ell$ is the $\ell$-th Adams operation \cite{FL} and $[\ell]:X \to X $ is multiplication
by $\ell$. The upper index notation means that we take the kernel of the corresponding endomorphism.

From Theorem~\ref{main.thm} we deduce:

\begin{cor}
A positive answer to Question~\ref{main.ques} would imply the Hodge conjecture for abelian varieties.
\end{cor}

\bigskip

\noindent
{\it Acknowledgment:} 
We would like to thank A.~Beilinson for many comments on our work on deformation of
cycles.

\section{ Milnor $K$-theory and differential forms}  \label{MilnorK}

\noindent
Let $k$ be a field of characteristic $0$ and write $S_n = \Spec k[t]/(t^{n})$. Let
$S=\Spec k[[t]]$, and let $X \to S$ be a smooth, separated scheme of finite type. 
Write $X_n = X\times_S S_n$. Write $\Omega^r_{X_n}$, $Z^r_{X_n}$ resp.\ $B^r_{X_n}$ for
the Zariski sheaf of absolute $n$-forms, closed absolute $n$-forms resp.\ exact
absolute $n$-forms on $X_n$. Let $\sKM_* $ be the Milnor $K$-sheaf with respect to Zariski
topology as studied in \cite{Ke}.

\begin{lem}\label{lem1} There is an exact sequence of Zariski sheaves
\eq{}{0 \to \Omega^{r-1}_{X_1} \xrightarrow{a} Z^r_{X_n} \xrightarrow{f} Z^r_{X_{n-1}}
}
where $a(\eta) = t^{n-1}d\eta + (n-1) t^{n-2}dt\wedge\eta$. 
\end{lem}
\begin{proof}Note that $a$ is well-defined because $t^{n}=t^{n-1}dt=0$ on $X_n$. Since 
\eq{}{\ker(\Omega^r_{X_n} \to \Omega^r_{X_{n-1}}) = t^{n-1}\Omega^r_{X_n}+t^{n-2}dt\wedge \Omega^{r-1}_{X_n}
}
the assertion is clear. 
 \end{proof}
\begin{lem}\label{lem2} 
There is an exact sequence of Zariski sheaves
\eq{3}{\Omega^{r-1}_{X_1} \xrightarrow{b} \sKM_{r,X_n} \to \sKM_{r,X_{n-1}} \to 0.
}
Here $b (x\, d\log(y_1)\wedge\cdots\wedge d\log y_{r-1}) = \{1+xt^{n-1},y_1,\ldots,y_{r-1}\}$. 
 \end{lem}
\begin{proof}We can present $\Omega^1_{X_1}$  with an exact sequence of $\sO_{X_1}$-modules 
\eq{}{0\to \sR \to \sO_{X_1}\otimes_\Z \sO_{X_1}^\times \xrightarrow{u\otimes v \mapsto udv/v} \Omega^1_{X_1} \to 0.
}
Here $\sR$ is the  sub-$\sO_{X_1}$-module under left multiplication with generators of the form $a\otimes a + b\otimes b -(a+b)\otimes (a+b)$. 
It follows that $\Omega^{r-1}_{X_1}= \bigwedge^{r-1}_{\sO_{X_1}}\Omega^1_{X_1}$ has a
presentation of the form
\eq{5}{ \sR\otimes_\Z \bigwedge_\Z^{r-2}\sO_{X_1}^\times \to \sO_{X_1}\otimes_\Z \bigwedge_\Z^{r-1}\sO_{X_1}^\times \to \Omega^{r-1}_{X_1} \to 0.
} 
For $r=2$ the map 
\eq{6}{\Omega^1_{X_1} \cong t^{n-1}\Omega^1_{X_n}\to \sKM_{2,X_n};\quad cda/a\mapsto \{1+ct^{n-1},a\}.
}
is well-defined. It boils down to showing $$\{1+at^{n-1},a\}+\{1+bt^{n-1},b\}-\{1+(a+b)t^{n-1},
a+b\}=0$$ in $\sK_{2, X_n}$ for $a, b, a+b$ in $\sO^\times_{X_1}$.
 See \cite[Sec.\ 2]{BlK2} for more details.
Then from the presentation \eqref{5} we deduce that $b$ is well-defined. The exactness of
\eqref{3} is straightforward. 
\end{proof}

\begin{prop}\label{lem3}  
The square 
\eq{7}{\begin{CD}\sKM_{r,X_n} @>>> \sKM_{r,X_1}  \\
@VV d\log V  @VV d\log V \\
Z^r_{X_n} @>>> Z^r_{X_1}
\end{CD}
}
is cartesian and there is a morphism between short exact sequences
\eq{sec1.shortex}{
\xymatrix{
0\ar[r]  & \Omega^{r-1}_{X_1}  \ar[r]^b \ar[d]  &  \sKM_{r,X_n}  \ar[r]\ar[d]^{d\log} &
\sKM_{r,X_{n-1}} \ar[r]\ar[d]^{d\log} & 0\\
0\ar[r]  & \Omega^{r-1}_{X_1}  \ar[r]^a   &  Z^r_{X_n}  \ar[r] & Z^r_{X_{n-1}}  \ar[r] & 0  
}
}
 \end{prop}
\begin{proof}
We have
\eq{}{\sKM_{r,X_1}\times_{Z^r_{X_1}}Z^r_{X_n} = \sKM_{r,X_1}\times_{Z^r_{X_1}} 
Z^r_{X_{n-1}}\times_{Z^r_{X_{n-1}}}Z^r_{X_n}
}
In order to prove the first statement
by induction we are thus reduced to proving that the diagram 
\eq{9}{\begin{CD}\sKM_{r,X_n} @>>> \sKM_{r,X_{n-1}} \\
@VV d\log V  @VV d\log V \\
Z^r_{X_n} @>>> Z^r_{X_{n-1}}
\end{CD}
}
is cartesian. Plugging in Lemmas \ref{lem1} and \ref{lem2} yields
\eq{10}{\begin{CD}@. \Omega^{r-1}_{X_1} @>>> \sKM_{r,X_n} @>>>\sKM_{r,X_{n-1}} @>>> 0 \\
@. @VV = V @VV d\log V @VV d\log V \\
0 @>>>\Omega^{r-1}_{X_1} @>>>     Z^r_{X_n} @>f>> Z^r_{X_{n-1}}  
      \end{CD}
}
It follows that \eqref{9} is cartesian and that the upper row in \eqref{sec1.shortex} is
exact  as claimed. The exactness of the lower row of \eqref{sec1.shortex}, i.e.\ the
surjectivity of $f$ in \eqref{10}, follows from the
commutative diagram of short exact sequences
\eq{}{\begin{CD}
       0 @>>> B^r_{X_n} @>>> Z^r_{X_n} @>>> \sH^r_{X_n} @>>> 0 \\
@. @VVV @VVV @VV \cong V \\
 0 @>>> B^r_{X_{n-1}} @>>> Z^r_{X_{n-1}} @>>> \sH^r_{X_{n-1}} @>>> 0
      \end{CD}
}
were $\sH^r$ is the de Rham cohomology sheaf. Indeed the right vertical map is an
isomorphism by Lemma~\ref{mil.lem4} and the left vertical map is obviously surjective.

\end{proof}

\begin{lem}\label{mil.lem4} With notation as above, the map of complexes $\Omega^*_{X_n} \to \Omega^*_{X_1}$ is a quasi-isomorphism.
\end{lem}
\begin{proof}The assertion is local so we may assume $X_n=\Spec A_n$ is affine. The algebra $A_1 = A_n/tA_n$ 
is therefore smooth over $k$, so by the infinitesimal criterion for smoothness \cite[(17.1.1)]{EGA4}, there exists a splitting of the surjection $A_n \to A_1$. It follows 
that we may write $X_n \cong X_1\times_k S_n$. 
This implies that there is an isomorphism between differential graded algebras $\Omega^*_{X_n}
\cong \Omega^*_{X_1}
\otimes_\Q \Omega^*_{S_n}$.  Since there is a short exact sequence
 $$ 0\to t k[t]/t^{n}
k[t] \xrightarrow{d} \Omega^1_{S_n} \to  \Omega^1_{k/\Q}\to 0 ,$$
we deduce that $\Omega^*_{A_n} \to \Omega^*_{A_0}$ is a quasi-isomorphism as claimed.
 \end{proof}

\section{Local cohomology}

\noindent
Let $X,S,X_n,S_n$ and $k$ be as in Section~\ref{MilnorK}. 
One of the central techniques for proving our main Theorem~\ref{main.thm} will be to study the coniveau
complex for Milnor $K$-sheaves of $X_n$. A general reference for the coniveau complex is \cite[Ch.~IV]{Ha}.

\begin{defn}\label{tr.defnres} \mbox{}
 For an arbitrary Zariski
sheaf of abelian groups $\sF$  on $X_1$ let us consider the coniveau complex of Zariski sheaves $\sC(F)$ defined as
\eq{tr.cpx}{\bigoplus_{x\in X_1^{(0)} } i_{x,*} H_x^0(X_1, \sF) \to  \bigoplus_{x\in X_1^{(1)}
  }i_{x,*}  H_x^1(X_1, \sF) \to  \bigoplus_{x\in X_1^{(2)}
  }i_{x,*}  H_x^2(X_1, \sF  )  \to \cdots . }
where the left group is put into cohomological degree $0$ and where $i_x:x\to X$ is the
natural monomorphism. There is a canonical
augmentation $\sF \to \sC^0(\sF)$.

By $\C(\sF)$ we denote the complex of global sections $\Gamma(X,\sC(\sF))$.
\end{defn}

\begin{defn}An abelian sheaf $\sF$ on $X_1$ is Cohen-Macaulay (CM for short) if for every
  scheme point $x\in X_1$ we have $H^i_x(X_1,\sF)=0$ for  $i\neq \codim(x)$. 
\end{defn}

 A basic observation about CM-sheaves is that they give rise to exact coniveau complexes.
 This follows directly from the degeneration of the coniveau spectral sequence for these
 sheaves, see~\cite[Prop.~IV.2.6]{Ha}. 
\begin{prop}\label{tr.bascm} Let $\sF$ be a CM abelian Zariski sheaf. Then $\sC(\sF)$ is an acyclic
  resolution of $\sF$. In particular one has $H^*(X,\sF)\cong H^*(\C(\sF))$. 
\end{prop}

The aim of this section is to show:
\begin{thm}\label{thm:CM}
The sheaves  $\sKM_{r,X_n}, \Omega^i_{X_n}, \sH_{X_1}^r, Z_{X_n}^r, B^r_{X_n}$ are CM for all $n\ge 1$.
\end{thm}

We prove the theorem in a series of propositions.

\begin{prop} \label{tr.res1}
\mbox{}
Let $r,n \ge 0$ be integers. The sheaves $\sF=\sKM_{r,X_n}$ and $\sF=\Omega_{X_n}^r$
are CM.
\end{prop}

\begin{proof}
For the sheaf  $\sF=\Omega_{X_1}^r$ see \cite[p.~239]{Ha}. For the sheaf  $\sF=\sKM_{r,X_1}$ see
\cite{Ke}.
We prove that  $\sF=\sKM_{r,X_n}$ is CM by induction on $n$. The case of the sheaf
$\sF= \Omega_{X_n}^r$ works similarly. For $n>1$ we get from the exact sequence  \eqref{sec1.shortex} an exact sequence
\[
 H_x^i(X_1, \Omega^r_{X_1})  \to H_x^i(X_1, \sKM_{r,X_n})  \to H_x^i(X_1, \sKM_{r,X_{n-1}}) .
\]
For $i \ne \codim(x)$ we already know that the groups on the left and the right side
vanish, so does the group in the middle.
\end{proof}

\medskip
Next we study the sheaf $Z^r_{X_1}$.

\begin{prop}\label{tr.blogcor}
For any $r\ge 0$ and  $x\in X_1$ the map
\[
 H^c_x(X_1,B^r_{X_1} ) \xrightarrow{} H^c_x(X_1,Z^r_{X_1} ) 
\]
induced by the inclusion
 $B^r_{X_1} \subset Z^r_{X_1}$
is surjective for $c\ne \codim(x) $ and injective for $c\ne \codim(x)+1$.
\end{prop}

\begin{proof}
The short exact sequence of sheaves
\[
0 \to B^r_{X_1} \to Z^r_{X_1} \to ,\sH^r_{X_1} \to 0
\]
induces a long exact cohomology sequence
\[
 \cdots\to  H^{c-1}_x(X_1,\sH^r_{X_1} ) \to H^c_x(X_1,B^r_{X_1} ) \to H^c_x(X_1,Z^r_{X_1} )
   \xrightarrow{} H^c_x(X_1,\sH^r_{X_1} ) \to \cdots
\]
On the other hand, we known from  \cite{BO} that the sheaves $\sH^r_{X_1}$ are CM. 
This shows the proposition.
\end{proof}

\begin{prop}\label{tr.CMclosed}
The sheaves $B^*_{X_1}, Z^*_{X_1}$ are CM. 
\end{prop}

\begin{proof}
 Using the exact sequence
\[
0 \to Z^r_{X_1} \to \Omega^r_{X_1} \to B^{r+1}_{X_1} \to 0
\]
and  Proposition~\ref{tr.res1} we get 
\eq{tr.eqvan2}{
 H^c_x(X_1,Z^{r}_{X_1} )= H^{c-1}_x(X_1,B^{r+1}_{X_1} )  \quad \text{ for }  c \notin\{
 \codim(x) , \codim(x)+1\}.
}
Combining Proposition~\ref{tr.blogcor} and \eqref{tr.eqvan2} we get for $c> \codim(x)$ the identifications
\begin{align*}
H^c_x(X_1,Z^r_{X_1} ) &\leftarrow H^c_x(X_1,B^r_{X_1} )  = H^{c+1}_x(X_1 , Z^{r-1}_{X_1})\\
&= H^{c+1}_x(X_1,B^{r-1}_{X_1} )=  H^{c+2}_x(X_1 , Z^{r-2}_{X_1})    = \cdots =0
\end{align*}
where the first arrow surjective
and for $c< \codim(x)$ the identifications
\begin{align*}
H^c_x(X_1 , B^r_{X_1} )&= H^c_x(X_1,Z^r_{X_1} ) = H^{c-1}_x(X_1,B^{r+1}_{X_1} ) =H^{c-1}_x(X_1,Z ^{r+1}_{X_1} ) \\
&= H^{c-2}_x(X_1,B^{r+2}_{X_1} ) = H^{c-2}_x(X_1,Z^{r+2}_{X_1} )   =\cdots  = 0 .
\end{align*}
\end{proof}

\begin{cor}\label{tr.corocm}
For all $n\ge 1$ the sheaves $B^*_{X_n}, Z^*_{X_n}$ are CM. 
\end{cor}

\begin{proof}
The sheaf $Z^*_{X_n}$ is CM by  Propositions \ref{lem3}, \ref{tr.res1} and
\ref{tr.CMclosed}.
Using Lemma~\ref{mil.lem4}
 and that $B^r_{X_n}$ trivially surjects onto $B^r_{X_{n-1}}$,
 we get a commutative diagram of short exact sequences
\eq{}{
\xymatrix{  0 \ar[r] & \Omega^{r-1}_{X_1} \ar[r] \ar@{=}[d] & B^r_{X_n} \ar[r] \ar[d] &
  B^r_{X_{n-1}} \ar[r] \ar[d] & 0 \\
0 \ar[r] &  \Omega^{r-1}_{X_1} \ar[r] &  Z^r_{X_n} \ar[r] & Z^r_{X_{n-1}} \ar[r] & 0
}}
which implies that $B^*_{X_n}$ is CM.
\end{proof}

\section{Transfer map}

\noindent
For any algebraic scheme $V$, it is natural to define the (cohomological) Chow groups 
\eq{}{\CH^p(V) := H^p(V, \sKM_p).
}
In this way, the graded object $\CH^*(V)$ is automatically a ring, contravariant in $V$.
For regular excellent $V$ the cohomology group $\CH^i(V)$ coincides with the usual Chow group of
codimension $i$ cylces on $V$ by \cite{Ke}.
 We would hope there should exist covariant transfer maps $\CH^*(V) \xrightarrow{f_*} \CH^*(W)$ of graded degree $-d$ for $f:V \to W$ smooth and proper with fiber dimension $d$.  One might further hope for $f$ proper and $d= \dim V -\dim W$ that there should exist a functorial map of coniveau complexes $\C(\sKM_{V,r}) \xrightarrow{f_*} \C(\sKM_{W,r-d})[-d]$ such that when $\sKM_{W,r-d}$ is CM one could define a covariant transfer via $H^p(V,\sKM_p) \to H^p(\C(\sKM_{p,V})) \xrightarrow{f_*} 
H^{p-d}(\C(\sKM_{W,p-d})) \cong H^{p-d}(W,\sKM_{W,p-d})$. 

In what follows we use results
from the previous section, together with work of Rost \cite{R} and Grothendieck \cite{Ha},
to define a transfer 
\[
f_* : \CH^i(X_n) \to \CH^{i-d}(Y_n),
\]
 for $f: X_n \to Y_n$ a smooth proper morphism of relative dimension $d$ between smooth schemes over $S_n$. This suffices to define a calculus of correspondences on $\CH^*(X_n)$, which is what we will need.

We use  the fiber square in Proposition~\ref{7} to `glue' the constructions of Rost and Grothendieck. From Propositions \ref{7} and \ref{tr.res1} and from Corollary \ref{tr.corocm} we obtain:

\begin{lem}\label{tr.lemcar2}
There is a Cartesian square of complexes
\[
\xymatrix{
\sC(\sKM_{r,X_n}) \ar[r] \ar[d]  & \sC(\sKM_{r,X_1}) \ar[d]  \\ 
 \sC(Z^r_{X_n})     \ar[r] & \sC(Z^r_{X_1}) 
}
\]
\end{lem}

Furthermore we get:

\begin{lem}\label{tr.lem12} We have a left-exact sequence
\[
0 \to \sC(Z^r_{X_n}) \to \sC(\Omega^r_{X_n}) \xrightarrow{d} \sC(\Omega^{r+1}_{X_n}) .
\]
\end{lem}
\begin{proof}A short-exact sequence of $CM$ sheaves yields a short-exact sequence of
  coniveau complexes. Applying this to the sequence $$0\to Z^r_{X_n} \to \Omega^r_{X_n}
  \to B^{r+1}_{X_n} \to 0,$$ we reduce to showing that the map $\sC(B^{r+1}_{X_n}) \to
  \sC(\Omega^{r+1}_{X_n})$ is injective. By the same logic, we know that
  $\sC(Z^{r+1}_{X_n}) \inj \sC(\Omega^{r+1}_{X_n})$, so it suffices to show
  $\sC(B^{r+1}_{X_n}) \inj \sC(Z^{r+1}_{X_n})$. 
By Lemma~\ref{mil.lem4} we have  an exact sequence
\eq{}{0 \to B^{r+1}_{X_n} \to Z^{r+1}_{X_n} \to \sH^{r+1}_{X_1} \to 0,
}
where $\sH^{r+1}_{X_1}$ is CM.  We conclude that $\sC(B^{r+1}_{X_n}) \inj \sC(Z^{r+1}_{X_n})$, proving the lemma. 
\end{proof}

Let now $f:X_n \to Y_n$ be as above a smooth proper map of relative dimension $d$ between smooth schemes over $S_n$.
Rost constructs  a morphism between complexes \cite{R} which, via the Gysin isomorphisms, is a transfer
\eq{tr.trrost}{
f_* \sC(\sKM_{r,X_1})  \xrightarrow{f_*} \sC(\sKM_{r-d,Y_1})[-d].
}
Grothendieck and Hartshorne construct in \cite{Ha} a morphism between complexes of
$\sO_{Y_n}$-modules
\eq{tr.trgro}{
f_* \sC(\Omega^d_{X_n/Y_n})  \xrightarrow{{\rm Tr}_f} \sC(\sO_{Y_n})[-d].
}
Indeed, using the notation of \cite{Ha} the structure sheaf $\sO_{Y_n}$ is pointwise dualizing, so its
coniveau complex $\sC(\sO_{Y_n})$ is a residual complex and $$f^! \sC(\sO_{Y_n}) =
\sC(\Omega^d_{X_n/Y_n}) [d].$$
Now we consider the composite morphism 
\begin{align}\label{tr.trdiffform}
f_* : & f_* \sC(\Omega^r_{X_n}) \to \Omega^{r-d}_{Y_n} \otimes_{\sO_{Y_n}} f_*
\sC(\Omega^d_{X_n/Y_n})  \xrightarrow{{\rm Tr}_f}\\\notag
 & \Omega^{r-d}_{Y_n}
\otimes_{\sO_{Y_n}} \sC(\sO_{Y_n})[-d] \xrightarrow{\sim} \sC(\Omega^{r-d}_{Y_n})[-d]
\end{align}
where for the first arrow we use  the projection $\Omega^r_{X_n}\to f^*\Omega^{r-d}_{Y_n}\otimes_{\sO_{X_n}} \Omega^d_{X_n/Y_n}$ followed by the projection formula.
Note that the composite map $f_*$ in \eqref{tr.trdiffform} is compatible with the differential.

One directly shows that the transfer map \eqref{tr.trdiffform} for $n=1$  is compatible
with the transfer map \eqref{tr.trrost}  with respect to
the $d\log$ map.
The transfer \eqref{tr.trdiffform} induces  thanks to Lemma~\ref{tr.lem12} a transfer
\eq{tr.grocl}{
\C(Z^r_{X_n})  \xrightarrow{f_*} \C(Z^{r-d}_{Y_n})[-d].
}

So we get a commutative diagram of exact sequences
\[\scriptsize
\xymatrix{
0 \ar[r] & \C(\sKM_{r,X_n})  \ar[r] \ar[d]^{f_*} &  \C(\sKM_{r,X_1}) \oplus  \C(Z^r_{X_n})
\ar[r] \ar[d]^{f_*} & \C(Z^r_{X_1})\ar[d]^{f_*}  \\
0 \ar[r] & \C(\sKM_{r-d,Y_n})[-d]  \ar[r] &  \C(\sKM_{r-d,Y_1})[-d] \oplus  \C(Z^{r-d}_{Y_n})[-d]
\ar[r] & \C(Z^{r-d}_{Y_1})[-d]
}
\]
The left vertical transfer map is defined by this diagram.
The transfer map 
\eq{}{
f_* : \CH^i(X_n) \cong H^i(\C(\sKM_{i,X_n}) ) \to H^{i-d}( \C(\sKM_{i-d,Y_n}) ) \cong \CH^{i-d}(Y_n),
}
obtained using the above construction, Proposition~\ref{tr.bascm} and Theorem~\ref{thm:CM},
 satisfies the usual properties, for example it is compatible with smooth
base change, and we use such properties without further mentioning.

\medskip

For the remaining part of this section let $d$ be the dimension of the smooth,
equidimensional scheme $X_n/S_n$. If one follows the above construction
of the transfer carefully,  one can deduce: 

\begin{prop}\label{tr.action}\mbox{}
\begin{itemize}
\item[(i)]
Composition of correspondences makes $\CH^d(X_n\times_{S_n} X_n)$ a ring with unity for
any $n\ge 1$.
\item[(ii)]
For $n\ge 2$ this ring acts canonically on the long exact cohomology sequence 
\[
\cdots \to  H^{c}(X_1,\Omega^{r-1}_{X_1}) \to H^c(X_n;\sKM_r)  \to H^c(X_{n-1};\sKM_r) \to
H^{c+1}(X_1,\Omega^{r-1}_{X_1})  \to \cdots 
\]
associated to \eqref{sec1.shortex} for any integer $r\ge 0$.
\item[(iii)]
$ \ker[ \CH^d(X_n\times_{S_n} X_n) \to \CH^d(X_1\times_{k} X_1) ]$ is a nilpotent ideal.
\end{itemize}
\end{prop}

\begin{prop} \label{tr.fundprop}
Fix a positive integer $i$ and assume condition ${\bf (CK)}_{X_\eta}^{2i}$ from the introduction.
Then there exists an inverse system of correspondences  $(\pi_n^{2i})_{n \ge 1}$ with
$\pi^{2i}_n \in \CH^{d}(X_n\times_{S_n} X_n)_\Q $ such that the following properties hold:
\begin{itemize}
\item[(i)] each  $\pi^{2i}_n$ is idempotent,
\item[(ii)] on $H^*_\dR(X_n/S_n)$ the correspondence $\pi^{2i}_n$ acts as the
  projection to $H^{2i}(X_n/S_n)$.
\end{itemize}
\end{prop}

\begin{proof}
  Consider $\pi^{2i}_\eta\in \CH^{d}(X_\eta\times_{K} X_\eta)_\Q $ as in property ${\bf
    (CK)}_{X_\eta}^{2i}$. The element $\pi_1^{2i}$ is defined as the specialization of
  $\pi^{2i}_\eta$ to the reduced closed fiber $X_1$. 
Recall that the specialization map 
\[
\CH^{d}(X_\eta\times_{K} X_\eta) \to \CH^{d}(X_1\times_{k} X_1)
\]
is a ring homomorphism of correspondences \cite[Sec.~20.3]{Fu84}.

   As a consequence, $\pi_1^{2i}$
  satisfies properties (i) and (ii)  of Proposition~\ref{tr.fundprop} for $n=1$. We claim that we can lift this
  element to an inverse system $(\pi^{2i}_n)_{n\ge 1}$ with the requested properties.

  We can extend $\pi^{2i}_\eta$ to an element $ \pi^{2i}\in
  \CH^{d}(X\times_{S} X)_\Q $. By the Gersten conjecture for the Milnor $K$-sheaf of the
  regular scheme $X\times_S X$ \cite{Ke} we
  get the left isomorphism in the following diagram
\begin{align*}
\iota_n :  \CH^{d}(X\times_{S} X) & \cong H^d(X\times_{S} X , \sKM_{d, X\times_{S} X })
  \longrightarrow \\
 &  H^d(X_n\times_{S_n} X_n , \sKM_{d, X_n\times_{S_n} X_n })  =
  \CH^d(X_n \times_{S_n} X_n).
\end{align*}
Now consider the inverse system of correspondences $\tilde\pi^{2i}_n = \iota_n(
\pi^{2i})$. They satisfy property (ii) of the proposition and furthermore $\pi_1^{2i}= \tilde\pi_1^{2i}$.  We will apply a
transformation to these correspondences which additionally makes them idempotent.

For an element $\alpha$ of an arbitrary (not necessarily commutative) unital ring and an
integer $s\ge 1$ set 
\[
f_s(\alpha) = \sum_{0\le j\le s}  \binom{2s}{j} \alpha^{2s-j}(1- \alpha)^j.
\]
From the argument in the proof of \cite[Prop.~III.2.10]{Ba} and from Proposition~\ref{tr.action}(iii) it follows that for $s$ large, depending on $n$, the element
$\pi_n^{2i}= f_s(\tilde\pi_n^{2i})$ is idempotent and independent of $ s\gg 0$. Observe that $f_s(
\pi_1^{2i})=\pi_1^{2i}$ for all $s\ge 1$, because $\pi_1^{2i}$ is idempotent. The elements
$(\pi_n^{2i})_{n\ge 1}$  form an inverse system of idempotent correspondences, finishing the proof of Proposition~\ref{tr.fundprop}. 
\end{proof}

In the next section we use  the action of a K\"unneth type correspondence on
cohomology of absolute differential forms as described in the following lemma.

\begin{prop}\label{tr.vanlem}
Consider for given $i\ge 0$ a correspondence $\pi^i\in \CH^d(X_1\times_k X_1)_\Q$ which acts on
$H_\dR^*(X_1/k)$  as the projection to $H_\dR^i(X_1/k)$. Then
\begin{itemize}
\item[(i)]
the action of the correspondence $(\pi^{i})^{r+1}$ on 
$
H^c(X_1,\Omega^r_{X_1})$ vanishes 
for all $c+r<i$.
\item[(ii)]
 the action of $(\pi^i)^{r}$ on 
\[
\ker[H^c(X_1,\Omega^r_{X_1}) \to H^c(X_1,\Omega^r_{X_1/k})   ]
\]
vanishes for $c+r=i$.
\end{itemize}
\end{prop}

\begin{proof}
The correspondence $\pi^i$ being algebraic respects the Hodge filtration on de Rham
cohomology and therefore acts on its graded pieces, which are Hodge cohomology groups according to
Hodge theory. Thus
$\pi^{i}$ acts trivially on $H^c(X_1,\Omega^r_{X_1/k})$ for  $c+r<i$. To pass from relative differential forms
to absolute differential forms we use the filtration
\[
L^s\Omega^r_{X_1} = \im[ \Omega_k^s \otimes_{k} \Omega^{r-s}_{X_1}  \to  \Omega^r_{X_1} ]
\quad (s=0, \ldots , r).
\]
Recall that
\[
L^s/L^{s+1}\Omega^r_{X_1} = \Omega^{r-s}_{X_1/k} \otimes_k  \Omega^s_k 
\]
and that $\pi^{i}$ acts on the system of morphisms
 \eq{tr.sys}{ H^c(X_1, \Omega^r_{k}\otimes_k \sO_{X_1} ) \to H^c(X_1, L^{r-1}\Omega^r_{X_1}) \to  H^c(X_1, L^{r-2}\Omega^r_{X_1})\to \cdots \to
 H^c(X_1, \Omega^r_{X_1}).
}
The filtration on cohomology  
$$L^s =  \im[ H^c(X_1, L^s\Omega^r_{X_1}) \to  H^c(X_1, \Omega^r_{X_1}) ]  $$
 has graded pieces $L^s/L^{s+1}$  which are subquotients
of $$H^c(X_1,\Omega^{r-s}_{X_1/k} ) \otimes_k \Omega^s_k.$$
On the latter groups the action of $\pi^{i}$ vanishes for $c+r-s<i$ as we have seen above.

This shows (i), while (ii) follows from the additional observation that
\[
L^1 = \ker[H^c(X_1,\Omega^r_{X_1}) \to H^c(X_1,\Omega^r_{X_1/k})   ].
\] 
\end{proof}

\medskip

\section{Deformation of Chow groups} \label{Chow}

\noindent
Let the notation be as in Section~\ref{MilnorK}.
We start this section with a basic lemma from \cite{Blo}  about the comparison of an obstruction map with a
Kodaira-Spencer map.

\begin{lem}\label{def.lem1}
The diagram
\[
\xymatrix{
H^i(X_{n+1} , Z^i_{X_{n+1}}) \ar[r] \ar[d]  & H^i(X_n , Z^i_{X_n})  \ar[r]^-{\rm Ob}
\ar[d]  &  H^{i+1}(X_1 , \Omega^{i-1}_{X_1}) \ar[d] \\
H^{2i}_\dR(X_{n+1}/S_{n+1})^\nabla\cap F^i \ar[r]   & H^{2i}_\dR(X_{n}/S_{n})^\nabla\cap F^i \ar[r]   \ar[r]^-{\rm KS}    &   H^{i+1}(X_1 ,   \Omega^{i-1}_{X_1/k})
}
\]
is commutative with exact rows.
\end{lem}

 Here $\nabla\in {\rm End}_{k}( H^*_\dR(X_n/S_n) ) $ is the Gauss-Manin connection and
 $F^i$ is the Hodge filtration.
The upper row is part of the long exact cohomology sequence
associated to \eqref{sec1.shortex} and $\rm KS$ is induced by the Kodaira-Spencer map \cite[(4.1)]{Blo}.

We can now state the version of our main theorem for (cohomological) Chow groups. Let
\[
\cl : \CH^i(X_n) \to H_\dR^{2i}(X_n/S_n)^\nabla \subset  H_\dR^{2i}(X_n/S_n)
\]
be the de Rham cycle class map, which is induced by the morphism of complexes
\[
d\log : \sKM_i [-i] \to \Omega^\bullet_{X_n/S_n}.
\]

The restriction map
\[
\Phi: H_\dR^{2i}(X/S)^\nabla \xrightarrow{\sim} H_\dR^{2i}(X_1/k)
\]
is an isomorphism by \cite[Prop.~8.9]{Katz}.

\begin{thm}\label{main.thmch}
Assume that for  $X/S$  as above and for a fixed $i$ the  property ${\bf
  (CK)}_{X_\eta}^{2i}$ explained in the introduction
holds. Then for $\xi_1 \in \CH^i(X_1)_\Q$ the following are equivalent:
\begin{itemize}
\item[(i)] $\Phi^{-1} \circ \cl(\xi_1) \in H^{2i}_\dR(X/S)^\nabla\cap F^i H^{2i}_\dR(X/S)$,
\item[(ii)] there is an element $\hat \xi \in (\varprojlim_n \CH^i(X_n))_\Q  $ such
  that 
\eq{cherneq}{\cl(\hat \xi |_{X_1}) = \cl (\xi_1) \in H_\dR^{2i}(X_1/k).}
\end{itemize}
\end{thm}

\begin{rmk}
For $k$ algebraic over $\Q$ it has been known to the experts for a long time,
(see \cite{GG}, and \cite{PaRa} for more recent work in the case of hypersurface sections), 
that for an element  $\xi_1 \in \CH^i(X_1)$ (note that we can use integral coefficients here)
condition (i) of the theorem is equivalent to:
\begin{itemize}
\item[(ii')] there is an element $\hat \xi \in \varprojlim_n \CH^i(X_n)  $ such
  that $ \xi |_{X_1 }= \xi_1$.
\end{itemize}
\end{rmk}

\begin{proof}[Proof of Theorem \ref{main.thmch}]
The implication (ii) $\Rightarrow$ (i) is clear. So consider (i) $\Rightarrow$ (ii).

\begin{claim}\label{chow.claim1}
The map
\[
 (\varprojlim_n \CH^i(X_n))_\Q
\to \varprojlim_n \CH^i(X_n)_\Q 
\]
is surjective.
\end{claim}

\begin{proof}
From the short exact sequence \eqref{sec1.shortex} we get for $n>1$ a commutative diagram with exact sequences
\[\xymatrix{
H^{i}(X_1, \Omega^{i-1}_{X_1}) \ar@{=}[d] \ar[r] & \CH^i(X_n) \ar[r] \ar[d] &
\CH^i(X_{n-1}) \ar[d] \ar[r]^-{\rm
  Ob} &
H^{i+1}(X_1, \Omega^{i-1}_{X_1}) \ar@{=}[d]\\
H^{i}(X_1, \Omega^{i-1}_{X_1})  \ar[r] & \CH^i(X_n)_\Q \ar[r] &  \CH^i(X_{n-1})_\Q \ar[r]^-{\rm
  Ob} &
H^{i+1}(X_1, \Omega^{i-1}_{X_1})
}\]
A diagram chase implies that the map 
\eq{eq.surchow}{\xymatrix{\CH^i(X_n) \ar@{->>}[r] & \CH^i(X_{n-1})
\times_{\CH^i(X_{n-1})_\Q } \CH^i(X_n)_\Q } }
 is surjective. From this Claim~\ref{chow.claim1} follows easily.
\end{proof}

By Claim~\ref{chow.claim1} it is enough to construct a pro-system $\hat\xi \in
\varprojlim_n \CH^i(X_n)_\Q  $ satisfying~\eqref{cherneq}. We will do this successively.

Choose correspondences $\pi_n^{2i}$ as in Proposition~\ref{tr.fundprop}. We claim that
there exists an element $$\hat \xi = (\hat \xi_n)_{n\ge 1} \in  \varprojlim_n \CH^i(X_n)_\Q  $$
such that 
\eq{def.eq1}{
\hat \xi_1 = \pi_1^{2i}\cdot \xi_1\quad \text{ and }\quad \pi_n^{2i}\cdot
\hat\xi_n=\hat\xi_n  \quad \text{  for all } n\ge 1.
}
 Indeed, assume we have already constructed $(\hat
\xi_m)_{1\le m\le n-1}$ with  property \eqref{def.eq1}. From Proposition~\ref{tr.action} we
know that $\pi_n^{2i}$ acts on the following diagram with exact row and column
\[
\xymatrix{  & &  K \ar[d]  \\
 \CH^i(X_n)_\Q  \ar[r] &\CH^i(X_{n-1})_\Q
\ar[r]^-{\rm Ob} &
H^{i+1}(X_1,\Omega^{i-1}_{X_1})  \ar[d]^{\sigma}   \\
  & & H^{i+1}(X_1,\Omega^{i-1}_{X_1/k})  
}
\]
where $K$ is defined as the kernel of $\sigma$. By Lemma~\ref{def.lem1} we have
$ {\rm Ob} (\hat\xi_{n-1})\in K $.
From the latter and Proposition~\ref{tr.vanlem} we deduce the third equality in
\eq{}{
{\rm Ob}(\hat \xi_{n-1}) = {\rm Ob}( (\pi_{n-1}^{2i})^{i-1} \cdot \hat\xi_{n-1}) =
(\pi_{1}^{2i})^{i-1} \cdot {\rm Ob} ( \hat\xi_{n-1}) =0. 
}
Because the obstruction vanishes we can find $\hat\xi'_n\in \CH^i(X_n)_\Q$ with
$\hat\xi'_n|_{X_{n-1}} = \hat\xi_{n-1}$. To finish the construction we set $\hat\xi_n =
\pi^{2i}_n \cdot\hat \xi'_n$.
\end{proof}

\section{Motivic complex and Chern character}\label{motcom}

\noindent
We begin this section by  proving that the canonical map from Milnor $K$-theory to Quillen
$K$-theory induces an isomorphism on certain relative $K$-groups. Then we study a Chern
character isomorphism using higher algebraic $K$-theory and motivic cohomology. For both
results the techniques are standard, so we only sketch the proofs.

 We consider a pro-system
of pairs of rings of the form
$(A_\bullet,A)$ with $A_\bullet= A[t]/(t^n)$ and we assume that $A$ is essentially smooth
over $k$ with ${\rm char} (k) = 0$.
By $K_*(R',R)$ we denote the relative Quillen $K$-groups of a homomorphism between rings $R'\to R$.

\begin{prop}\label{com.main}
For $A$ as above the
canonical homomorphism  
\eq{com.mapcom}{
\ker [\rKM_*(A_\bullet) \to  \rKM_*(A)] \xrightarrow{\sim} K_* (A_\bullet,A) 
}
is a pro-isomorphism.
\end{prop}

\begin{proof}
By  Goodwillie's theorem \cite{Good} there is an isomorphism
\eq{com.trmap}{
K_{i+1} (A_n,A)  \xrightarrow{\sim} \rHC_{i}(A_n,A)
}
for any $n\ge 1$. There is a canonical homomorphism
\eq{com.en}{
e_n: \rHC_i(A_n)  \to \Omega^i_{A_n}/B^i_{A_n} \oplus Z_{A_n}^{i-2}/B_{A_n}^{i-2} \oplus
Z_{A_n}^{i-4}/B_{A_n}^{i-4}  \oplus \cdots,
}
see \cite[9.8]{W}.  
By the Hochschild-Kostant-Rosenberg theorem \cite[Thm.\ 9.4.7]{W} resp.\ a pro-version of
it (see \cite{Mor2} for a general discussion), one sees that the
corresponding maps on Hochschild homology
\[
e: {\rm HH}_*(A) \to \Omega_{A}^* \quad\text{ resp. }\quad   e: {\rm HH}_*(A_\bullet) \to \Omega_{A_\bullet}^*
\]
induce an isomorphism resp.\ a pro-isomorphism. Then by a short argument with mixed
complexes \cite[9.8.13]{W} one deduces that the map $e_1$ and the pro-system of maps $e_\bullet$ from
\eqref{com.en} induce an isomorphism resp.\ a pro-isomorphism. Finally, using
Lemma~\ref{mil.lem4} we see that we get pro-isomorphisms
\eq{com.eq1}{
\rHC_i(A_\bullet,A)  \xrightarrow{e_\bullet} \ker[
\Omega^i_{A_\bullet}/B^i_{A_\bullet}\to \Omega^i_{A}/B^i_{A}]  \xrightarrow{d}  \ker[
Z^{i+1}_{A_\bullet} \to  Z^{i+1}_{A} ].
}
Following the steps of this construction carefully shows that the composition of
\eqref{com.mapcom}, \eqref{com.trmap} and \eqref{com.eq1} is equal to the $d\log$ map, which is an isomorphism
by the cartesian square~\eqref{7}. 
\end{proof}

More general results in the direction of Proposition~\ref{com.main} can be found in
\cite{Mor} and \cite{Mor2}.

By classical techniques  one constructs a Chern character ring homomorphism
\eq{com.chiso}{
\ch: K_0(X_n)_\Q  \to  \bigoplus_{i\ge 0} \CH^i(X_n)_\Q,
}
where we use the notation of Section~\ref{MilnorK}. The Chern character  is characterized by
the property that the composite morphism 
\[
H^1(X_1 , \sO^\times_{X_{n}} )\cong\Pic(X_n) \to  K_0(X_n)_\Q \xrightarrow{\ch} \bigoplus_{i\ge 0}
\CH^i(X_n)_\Q \to \CH^1(X_n)_\Q 
\] 
is induced by the canonical isomorphisms $ H^1(X_1, \sO^\times_{X_n}) \cong \CH^1(X_n) $.

Using Proposition~\ref{com.main} and Zariski descent for algebraic $K$-theory \cite[Sec.~10]{TT} we
will show in this section that \eqref{com.chiso} induces a pro-isomorphism with respect
to $n$, see Theorem~\ref{com.mainthm}.
 The pro-isomorphism \eqref{com.chiso} together with 
Theorem~\ref{main.thmch} immediately imply Theorem~\ref{main.thm}.

Let $\Z_{X_1}(r)$ be the weight $r$ motivic complex of Zariski sheaves on the smooth variety $X_1/k$ constructed by
Suslin-Voevodsky, see \cite{MVW}.
We define a motivic complex $\Z_{X_n}(r)$ of the scheme $X_n$ by the homotopy
Cartesian square
\eq{motcompdef}{
\xymatrix{
\Z_{X_n}(r) \ar[r] \ar[d]  & \Z_{X_1}(r)  \ar[d]\\
  \sKM_{r,X_n}[-r] \ar[r]   & \sKM_{r,X_1}[-r]
}}
The reader should be warned that the complex $\Z_{X_n}(r)$ for fixed $n$ cannot be `the
correct' motivic complex of $X_n$, but as a pro-system in $n$ we get the `right' motivic
theory. In fact from comparison with algebraic $K$-theory we expect the `proper' homotopy fiber of
the upper row of \eqref{motcompdef} to have non-trivial cohomology sheaves in each degree
in the interval $[1,r]$ and not only in degree $r$.

There is a Chern character homomorphism from (higher) algebraic $K$-theory to the
cohomology of the motivic complex $\Z_{X_\bullet}(r)$.
The technique of the construction of the higher Chern character is explained in \cite{Gil}. We recall the construction.

The universal Chern character 
\[
\ch \in \bigoplus_{r\ge 0} H^{2r}({\rm BGL}_k, \Q_{X_1}(r))
\]
induces morphisms
\eq{com.ch1}{
\ch_r : K_{X_\bullet} \to \mathfrak K  \Q_{X_\bullet}(r)[2r]
}
in the homotopy category of pro-spectra in the sense of \cite{Isa}. Here $K(X_n)$ is the non-connective $K$-theory
spectrum of $X_n$ \cite[Sec.~6]{TT} and $\mathfrak K $ is the Eilenberg--MacLane functor.

Moreover,  $\ch$ induces morphisms of Zariski descent spectral sequences \cite[Sec.~10]{TT}
\begin{align*}
^K \!E_2^{i,j}(X_\bullet) = H^i(X_1, \sK_{-j,X_\bullet}) &\Rightarrow
K_{-i-j,X_\bullet}   \\
^{\rm mot}\! E_2^{i,j}(X_\bullet) = H^i(X_1, \sH^j(\Z_{X_\bullet}(r) )) &\Rightarrow  H^{i+j}(X_1,
\Z_{X_\bullet}(r))
\end{align*}
of the form
\eq{com.ch12}{
\ch_r : ^K \!\! E_2^{i,j}(X_\bullet)_\Q \to ^{\rm mot}\!\! E_2^{i,j+2r}(X_\bullet)_\Q.
}
For any $r\ge 0$ there is a similar Chern character of relative theories
\eq{com.ch2}{
\ch_r : ^K\!\! E_2^{i,j}(X_\bullet,X_1) \to ^{\rm mot}\!\! E_2^{i,j+2r}(X_\bullet,X_1).
}
which is a pro-isomorphism for $r=-j$ by Proposition~\ref{com.main} and vanishes
otherwise.
It is well-known \cite{Bl} that the Chern character induces an isomorphism 
\eq{com.ch3}{
\ch : K_i(X_1)_\Q \xrightarrow{\sim} \bigoplus_{r\ge 0} H^{2r-i}(X_1, \Q_{X_1}(r))
}
for any $i\in \Z$.
Combining isomorphisms \eqref{com.ch2} and \eqref{com.ch3} we get
 the
required isomorphism between pro-groups:

\begin{thm}\label{com.mainthm}
For any smooth scheme $X/S$, which is separated and of finite type, there is a pro-isomorphism 
\eq{com.chiso2}{
\ch : K_i(X_\bullet)_\Q \xrightarrow{\sim} \bigoplus_{r\ge 0} H^{2r-i}(X_1, \Q_{X_\bullet}(r)).
}
\end{thm}

Observing that for any $n\ge 1$ there is a canonical isomorphism
\eq{}{
H^{2r}(X_1, \Z_{X_n}(r)) \xrightarrow{\sim} \CH^r(X_n)
}
we see that the pro-isomorphism \eqref{com.chiso2} comprises the pro-isomorphism \eqref{com.chiso}.

\begin{appendix}
 \section{Two versions of Grothendieck's conjecture} \label{groth.conj}

\noindent
In the introduction we stated  as Conjecture~\ref{int.conj}  a local version  of Grothendieck's principle of the parallel transport of
cycles, which we will refer to as the {\em infinitesimal Hodge conjecture} in
the following. His original formulation, today called {\em variational Hodge conjecture}, is more global and we show in this appendix
that the two formulations are equivalent.

 \medskip
 
Let $k$ be a field of characteristic $0$.
  Let $f: \sX\to \sS$ be a smooth projective morphism, where $\sS/k$ is a smooth variety.
  Fix a point $s\in \sS$ and let $\sX_s$ be the fiber over $s$.

Grothendieck's original conjecture \cite[p.~103]{GdR}  can now be stated as follows.
\begin{conj}[Variational Hodge]\label{ap.conjgro}
For $\xi_s\in K_0(\sX_s)_\Q$ the following are equivalent:
\begin{itemize}
\item[(i)] $\ch(\xi_s)\in H_\dR^*(\sX_s/s)$ lifts to an element  in $H_\dR^*(\sX/k) $,
\item[(ii)] there is an element $\xi\in K_0(\sX)_\Q$ with $\ch(\xi|_{\sX_s}) = \ch(\xi_s)$.
\end{itemize}
\end{conj}

\begin{prop}\label{ap1.prop}
The variational Hodge conjecture (Conjecture~\ref{ap.conjgro}) for all $k,\sX,\sS$ as
above is equivalent to the infinitesimal Hodge conjecture (Conjecture~\ref{int.conj}) for all $k,X$ as in the introduction.
\end{prop}

\begin{rmk}
The same proof shows that the variational Hodge conjecture for abelian schemes $\sX/\sS$ is
equivalent to the infinitesimal Hodge conjecture for abelian schemes $X/S$.
\end{rmk}

\begin{proof} 
\hspace{10pt}
{\em Infinitesimal Hodge $\Longrightarrow$ Variational Hodge:}

By induction on $\dim(\sS)$ we will reduce to $\dim(\sS)=1$. In order to do this
observe first that we can replace without loss of generality $\sS$ by a dense open
subscheme containing $s$ and $\sX$ by the corresponding pullback. 

Now let $\xi_s$ satisfy  Conjecture~\ref{ap.conjgro}(i) and assume without loss of
generality that $\codim(s)>0$.  Choose a smooth
hypersurface $\sS'\subset \sS$ containing $s$, which exists after possibly replacing $\sS$ by a dense
open subscheme, and set $\sX' = \sX \times_\sS \sS'$. By the induction
assumption there is $\xi'\in K_0(\sX')_\Q$  with  $\ch(\xi'|_{\sX_s}) = \ch(\xi_s)$.

 Let $s'$ be the generic point of $\sS'$ and choose an extension of fields $k\subset
 k'\subset k(s')$ such that the second inclusion is finite and such that there exists a
 lift $k' \to \sO_{\sS,s'}$. This lift gives rise to a curve $\sS''/k'$ mapping to $\sS$
 as schemes over $k$ such that $s'$ is contained in the image. Now one applies the one-dimensional case of
 Conjecture~\ref{ap.conjgro} to the family $$\sX''=\sX \times_{\sS} \sS''\to \sS'' \text{ with
   class } \xi'|_{s'}\in K_0(\sX_{s'})_\Q $$
to get a lifted class  $ \xi''\in K_0(\sX'')_\Q$. Finally, $\sX''\subset \sX$ is an
inverse limit of open immersions of regular schemes, so we can extend $\xi''$ to a class $\xi\in
K_0(\sX)_\Q$, which will then satisfy the requested Conjecture~\ref{ap.conjgro}(ii).

\medskip

Now we assume $\dim(\sS)=1$. Without loss of generality $k=k(s)$. Using Deligne's
partie fixe \cite[4.1]{DH2} we can assume without loss of
generality that the lift $\alpha$ of $\ch(\xi_s)$ in Conjecture~\ref{ap.conjgro}(i) lies
in the image of
\[ 
\bigoplus_i H^{2i}(\sX , \Omega^{\ge i}_{\sX /k})  \to H_\dR^*(\sX/k).
\]
 The completion of
$\sO_{\sS,s}$ along the maximal ideal is isomorphic to $k[[t]]$. So define $X=
\sX \times_\sS S$, where $S=\Spec k[[t]]$ and apply Conjecture~\ref{int.conj} to the class
$\xi_1=\xi_s$ and the flat lift $\alpha|_X$ of $\ch(\xi_1)$ to get a lifted class $\dot\xi \in K_0(X)_\Q$. 

Denote by $S^{\rm h}$ the spectrum of the henselization of $\sO_{\sS,s}$ and by $X^{\rm h}$ the pullback
$\sX \times_\sS S^{\rm h}$. By Artin approximation \cite{Ar} there is a class $\ddot \xi \in
K_0(X^{\rm h})_\Q$ with $\dot \xi |_{\sX_s} = \ddot \xi|_{\sX_s}$. By a standard transfer
argument we get from $\ddot \xi$  a class $\xi \in K_0(\sX)_\Q$ with the requested
property of Conjecture~\ref{ap.conjgro}(ii).  

\medskip

\hspace{10pt}
{\em  Variational Hodge   $\Longrightarrow$   Infinitesimal Hodge:}

Let $\xi_1$ satisfy property (i) of Theorem~\ref{main.thm}.
The idea is roughly the following:
\begin{itemize}
\item[(1)] 
reduce to a situation where  $X\to S$
`extends' to a morphism between complex varieties $\sX \to \sS$, 
\item[(2)]
 use complex Hodge
theory in order to show that $\ch(\xi_1)$ extends as a de
Rham cohomology class to $H^*_{\dR}(\sX/\CC)$ so that we can apply
Conjecture~\ref{ap.conjgro}.
\end{itemize}

By a simple reduction we can assume without loss of generality that $k$ contains $\CC$. 
 Then as step (1) we find a local subring $R\subset k[[t]]$ with maximal ideal $m$ which is essentially of finite
type over $\Q$ and such that $X$ descends to a projective smooth scheme $X_R$ over $R$ and
such that
$\xi_1$ descends to a class in $K_0(X_R \otimes R/m)_\Q$. Such an $R$ exists by the
techniques of \cite[Sec.~IV.8]{EGA4}. By resolution of singularities we can assume
without loss of generalities that $R$ is regular. Choose a subfield $k' \subset R $ such
that the field extension $k' \subset R/m$ is finite and such that $k'$ is algebraically
closed in $R$. Now we can extend $$ \Spec (X_R
\otimes_{k'} \CC) \to \Spec ( R\otimes_{k'} \CC)$$
 to a smooth projective morphism $f:\sX \to \sS$ of smooth varieties over $\CC$.
There is a canonical morphism $\gamma: S \to \sS$ over $\CC$.
In other words we get a cartesian square 
\[
\xymatrix{
X \ar[r]  \ar[d] & \sX \ar[d]^f \  \\
S \ar[r]_\gamma  &  \sS 
}
\]
 The map $\gamma$ maps the closed point of $S$ to a closed point $s\in \sS$ and the
 generic point of $S$ to the generic point of $\sS$.
 There is an induced class $\xi_s \in K_0(\sX_s)_\Q$, which
$\gamma$ pulls back to our originally given class $\xi_1\in K_0(X_1)_\Q$.

\medskip

We claim (step (2)) that the de Rham class $\ch(\xi_s) \in H_\dR^*(\sX_s/\CC) $ extends
to  $H^*_\dR(\sX/\CC)$ after possibly replacing $\sS$ by an \'etale neighborhood of $s$. This will allow us to apply  Conjecture~\ref{ap.conjgro} to
obtain a class $\dot \xi \in K_0(\sX)_\Q$, so that the requested class  in $K_0(X)_\Q$ from
Conjecture~\ref{int.conj}(iii) is given by $\xi=\gamma^*(\dot \xi)$.
This will finish the proof of Proposition~\ref{ap1.prop}.

To show the claim let $f^\CC:\sX(\CC) \to \sS(\CC)$ be the induced map of complex
manifolds and consider the local system $L=R^*f^\CC_* \Q$ on $\sS(\CC)$ (we omit any Tate
twists). We think of $L$ as an \'etale manifold over $\sS(\CC)$. In this sense let $L_0$ be
the connected component of the Betti Chern character class $\ch(\xi_s) \in H^*_{\rm
  B}(\sX_s(\CC))$ in the unramified complex space $$L \cap \bigoplus_{i\ge 0}  R^{2i} f^\CC_*
(\Omega^{\ge i}_{\sX/\sS}) $$ over $\sS(\CC)$.

By \cite{CDK} we know that $L_0$ is finite over $\sS(\CC)$ and therefore given by a finite
unramified scheme over $\sS$, which we denote by the same letter. After replacing $\sS$ by
an \'etale neighborhood of $s$ we can therefore assume that $L_0\to \sS$ is a closed immersion. 

 Clearly, $\gamma^*(L_0)$
contains the locus where $\Phi^{-1} \circ \ch(\xi_1)$ lies in the Hodge filtration
\[
\bigoplus_{i\ge 0} F^iH_{\dR}^{2i}(X/S)
\]
 (the
map $\Phi$ is defined in \eqref{int.solv}). By our assumption on $\xi_1$ this locus is all
of $S$. So we get that $L_0 \to \sS$ is an isomorphism, since $\gamma$ has dense image.
This means that the monodromy action of $\pi_1(\sS(\CC),s)$ on $H^*_{\rm
  B}(\sX_s(\CC))$ fixes $\ch(\xi_s)$. By the degeneration of the Leray spectral sequence \cite[4.1]{DH2}  the cohomology
class $\ch(\xi_s) \in H_\dR^*(\sX_s/\CC)$ extends to $H_\dR^*(\sX/\CC)$, proving the claim.

\end{proof}

\section{A counterexample to algebraization}   \label{counterex}

\noindent
In this section, we show that algebraization of $K_0$-classes of vector bundles does not
hold in general, i.e.\ the map~\eqref{intr.algmap} is usually `far' from beeing an
isomorphism. For a precise statement see Proposition~\ref{prop:ex}. The idea is to consider a `pro-$0$-cycle' on the trivial deformation over $\CC[[t]]$ of a
smooth projective variety $Y/\CC$  whith $p_g >0$. Roughly speaking we construct such a pro-$0$-cycle whose
top Chern class in absolute Hodge cohomology `jumps' around so much in the pro-system that it cannot  come from absolute
differential forms on $Y \otimes_\CC \CC[[t]]$.

\smallskip

 We start the discussion by certain elementary observations about absolute differential forms. One defines a weight function
on differential forms $\tau \in \Omega^2_{\CC/\Q}$ by
 \ga{}{  w(\tau):= \min\{n\ |\ \tau = \sum_{i=1}^n a_idb_i\wedge dc_i, \ a_i, b_i, c_i \in \CC .\}  \notag}
 The function 
 $w$ is subadditive in the sense that 
 \ga{}{w(\tau_1+\ldots + \tau_p) \le \sum w(\tau_i). \notag}
\begin{lem}\label{lem1d} Let $\tau = \sum_{i=1}^n db_i\wedge dc_i$ and assume all the $b_i, c_i$ are algebraically independent elements in $\CC$. Then $w(\tau) = n$.
\end{lem}
\begin{proof}Clearly $w(\tau) \le n$. If $w(\tau) <n$ we can write 
\eq{1c}{\tau = \sum_{j=1}^{n-1}\alpha_jd\beta_j\wedge d\gamma_j.
}

The $n$-fold wedge $\wedge^n\tau=\tau\wedge\cdots\wedge\tau \in \Omega^n_{\CC}$ is equal to  $n!db_1\wedge dc_1\wedge db_2\wedge\cdots\wedge dc_n$, which is non-zero in $\Omega^{2n}_{\CC}$ as the $a_i, b_j$ are algebraically independent.
On the other hand, if \eqref{1c} holds, then $\wedge^n\tau = 0$ in $\Omega^{2n}_{\CC}$, a contradiction.
\end{proof}

Let $R=\CC[[t]]$ and write $R_n = R/t^nR$.

The choice of a parameter $t$ yields a natural splitting
\ga{sp}{  \Omega^i_{\CC}\otimes_{\CC} R_n \to \Omega^i_{R_n} \to  \Omega^i_{\CC}\otimes_{\CC} R_n  }
which is compatible in the pro-system in $n$. Thus it defines a homomorphism
\ga{cp}{\Omega^2_{R}\to \varprojlim_n  \Omega^2_{\CC}\otimes_{\CC} R_n =: \Omega^2_{\CC}\widehat \otimes R\\
\gamma = \sum_{i=1}^N f_i(t)dg_i(t)\wedge dh_i(t) \mapsto  \tilde{\gamma}=\sum_{p=1}^\infty t^p\tau_p,  \notag \\ f_i, g_i, h_i \in R, \ \tau_p\in \Omega^2_{\CC}.  \notag}

\begin{lem}\label{lem1b} One has $w(\tau_p) \le N\binom{p+2}{p}$. 
\end{lem}
\begin{proof}Suppose first $N=1$ and $\omega = fdg\wedge dh$. Write $f=\sum_{j=0}^\infty f^{(j)}t^j$ and similarly for $g, h$. Then 
\eq{}{fdg\wedge dh = \sum_p t^p\sum_{i+j+k=p}f^{(i)}dg^{(j)}\wedge dh^{(k)}. 
}
the inner sum has $\binom{p+2}{p}$ terms, and the result follows in this case simply by the definition of $w$. For $N$ general, we conclude by subadditivity of $w$. 
\end{proof}

\begin{lem}\label{lem3c}
 Let 
 \ga{}{ \tilde\eta =   \sum_{p=0}^\infty t^p\eta_p \in   \Omega^2_\CC  \widehat\otimes\CC[[t]]. \notag}
  Assume 
\[\limsup_{p\to \infty} \frac{ w(\eta_p)}{p^2} = \infty. 
\]
Then $\tilde\eta$ does not lift  via \eqref{cp} to an element in $\Omega^2_{\CC[[t]]}$. 
\end{lem}
\begin{proof}Immediate from Lemma~\ref{lem1b}. 
\end{proof}

\begin{rmk}
All the above lemmas of this section immediately generalize to differential forms of any
even degree.
\end{rmk}

\smallskip

 Let $Y/\Q$ be a smooth projective variety. We 
write $Y_A$ for the base change by a ring $A/\Q$.

\begin{prop} \label{prop:ex} Assume that $p_g= \dim_{\Q} H^0(Y,\Omega^{\dim Y}_Y)>0$ and
  that the dimension of $Y$ is even.
Then the map  
 \ga{}{K_0(Y_{R})_\Q \to (\varprojlim_n K_0(Y\times_{\CC}  R_n))\otimes \Q \notag}
  is not surjective.
\end{prop}

For simplicity of notation we restrict to $\dim Y=2$ for the rest of this section. The
proof of the general case of Proposition~\ref{prop:ex} works exactly the same way.
For $A$ a ring over $\overline\Q$, we have a second Chern character in abolute Hodge cohomology 
\eq{}{{\rm ch}_2: K_0(Y_A) \to H^2(Y_A, \Omega^2_{Y_A}). \notag
}
Using the K\"unneth decomposition for differential forms and the resulting projection
\ga{}{H^2(Y, \Omega^2_{Y_A})\to H^2(Y, \sO_Y)\otimes \Omega^2_A \notag}
one defines by composition
\eq{}{\widetilde {\rm  ch}_2:  K_0(Y_A) \to H^2(Y_A, \Omega^2_{Y_A})\to H^2(Y,\sO_Y)\otimes \Omega^2_A .
}
 Taking above $A$ to be $R_n$ and composing with the projection in \eqref{sp}, one obtains  the homomorphism
\eq{4b}{\overline {\rm ch}_2: K_0(Y_{R_n}) \to   H^2(Y, \sO_Y)\otimes_{\overline \Q} \Omega^2_\CC\otimes_{\CC} R_n .
}

\smallskip

For the following discussion we choose a point $z \in Y(\overline\Q)$ and generators $t_1,t_2$
of the maximal ideal of $\sO_{Y,z}$. Write $X=\Spec \sO_{Y,z}$ and $U=X\setminus z$.

This choice gives rise to an element $\rho \in H^2(Y,\sO_{Y})$ by the following
construction. In fact for later reference we explain the construction after performing a base change to $\CC$. Consider the covering $\mathcal U=(U_i)_{i=1,2}$ of $U_\CC$ with $U_i=X_\CC
\setminus V(t_i)$. Now $\rho$ is the image of the \v{C}ech cocycle 
$\frac{1}{t_1 t_2} \in \check{H}^1(\mathcal U, \sO_{U_\CC})$ under the
composite map
\begin{align}\label{appb.cech}
H_n^r : \check{H}^1(\mathcal U, \Omega^r_{U_{R_n}}) \to H^1(U_\CC, \Omega^r_{U_{R_n}}) \to H^2_z(X_\CC,
\Omega^r_{X_{R_n}}) \xleftarrow{\sim}\\\notag 
H^2_z(Y_\CC, \Omega^r_{Y_{R_n}}) \to H^2(Y_\CC, \Omega^r_{Y_{R_n}}) 
\end{align}
for $n=1$ and $r=0$.

\begin{lem}\label{appb.gen}
Assume $p_g={\rm dim}_{\bar \Q} H^0(Y, \Omega^2)> 0$. Then a generic choice of $z$
gives rise to non-vanishing $\rho\in H^2(Y,\sO_{Y})$.
\end{lem}

\begin{proof}
Choose a non-vanishing $\omega \in H^0(Y,\Omega^2_Y)$. Duality theory \cite{Ha} shows that
$\omega \cup \rho \in H^2(Y,\Omega^2_{Y})\cong  \overline \Q$ does not vanish if $\omega$ does not
vanish at $z$.  
\end{proof}

\begin{lem}\label{lem4}  For $n\ge 1$, the image of 
\ml{5b}{\ker \big(K_0(Y_{R_{n+1}}) \to K_0(Y_{R_n})\big) \cap \im \big( K_0(Y_{R}) \to  K_0(Y_{R_{n+1}})
   \big)  \xrightarrow{\overline{ch}_2}\\
  H^2(Y, \sO_Y)\otimes_{\overline \Q} \Omega^2_\CC\otimes_{\CC} R_{n+1} }
contains any element of the form $\rho \otimes t^{n}(da_1\wedge db_1 + \cdots + da_p\wedge
db_p)$ with $a_i,b_i\in\CC$.
\end{lem}
\begin{proof}
It  suffices to show that any element of the form $\rho\otimes t^{n}da\wedge db$ lies in
the image of the map \eqref{5b}.

Accoring to Grothendieck--Riemann--Roch \cite[Ex.~15.2.15]{Fu84}  one has:
\begin{claim} \label{Kth} \mbox{}
\begin{itemize}
\item[(i)]
For any $R_{n+1}$-point $x\in Y(R_{n+1})$ there is a canonical pushforward $\Z = K_0(R_{n+1})\to
K_0(Y_{R_{n+1}})$. We denote the image of $1$ under this map by $[x]$.
\item[(ii)]
Assume that $x$ as in (i) lifts the point $z$. Then
${\rm ch}_2([x])$ is equal to $$H_{n+1}^2(d\log (t_1-x^*(t_1)) \wedge d\log (t_2-x^*(t_2))
),$$ where the
map $H^2_{n+1}$ is as defined in~\eqref{appb.cech}.
\end{itemize}
\end{claim}

For $a,b\in \CC$ we consider  two $R_{n+1}$-points $x,y \in Y(R_{n+1})$ specializing to $z$ which are
defined by
 \ga{}{x: t_1 \mapsto at^{n-1},\ t_2 \mapsto bt; \\
 y: t_1 \mapsto 0;\ t_2 \mapsto bt.
}
Observe that the pullbacks of $[x], [y]$ to $K_0(X_{R_n})$ coincide. It is
clear that $x$ and $y$ extend to $R$-points of $Y$. So $[x]-[y]$ lies in the group on the
left side of \eqref{5b}. The following claim  shows that $$\overline{\rm ch}_2([x]-[y]) =
\rho\otimes t^{n}da\wedge db$$
finishing the proof of Lemma~\ref{lem4}. 
\end{proof}

\begin{claim} \label{appb.formula}
\ga{appb.f1}{\overline{{\rm ch}}_2([x]) = \rho \otimes t^{n}da\wedge db  \\
\overline{{\rm ch}}_2([y]) =  0. }
\end{claim}
\begin{proof}
We give the proof for $x$, the case of $y$ works similarly. In the \v{C}ech cohomology
group $\check{H}^2(\mathcal U, \sO_U  )\otimes_{\overline \Q} \Omega^2_\CC \otimes_\CC R_{n+1}$ we
have
\begin{align*}
d\log &(t_1-x^*(t_1)) \wedge d\log (t_2-x^*(t_2)) = \frac{d x^*(t_1) \wedge d
  x^*(t_2)}{(t_1-x^*(t_1)) (t_2-x^*(t_2))}\\
&= \frac{t^n da\wedge db}{(t_1-x^*(t_1)) (t_2-x^*(t_2))}
= \frac{t^n da\wedge db}{t_1 t_2}.
\end{align*}
So \eqref{appb.f1} follows from Claim~\ref{Kth}(ii) and the definition of $\rho$. 
\end{proof}
Let $K=\CC((t))$. 
\begin{lem}\label{lem1a} One has $K_0(Y_R) \cong K_0(Y_K)$. 
\end{lem}
\begin{proof}The boundary $\partial:K_1(Y_K) \to K_0(Y_k)$ is surjective: as $Y_R$ admits a morphism $Y_R\to Y_k$,  one applies the formula $x=\partial(x_R\cdot [t])$ where $x_R \in K_0(Y_R)$ is the pullback of $ x \in K_0(Y_k)$ via the projection $Y_R \to X_k$ and  $[t] \in K_1(K)$ is the class of the unit $t\in K^\times$. 
 It follows that $K_0(Y_R) \inj K_0(Y_K)$. For surjectivity, it suffices to note that a coherent sheaf on $Y_K$ can be extended to a coherent sheaf on $Y_R$, and, as $Y_R$ is regular, it can be resolved by locally free sheaves. 
\end{proof}

\begin{proof}[Proof of Proposition~\ref{prop:ex}] Recall that for simplicity of notation
  we assume $\dim Y=2$.   The diagram
\eq{13}{\begin{CD} K_0(Y_R) @>(3) >> \varprojlim_n K_0(Y\times_{\CC}  R_n) \\
@VV\widetilde{{\rm ch}}_2 V @VV\overline{{\rm ch}}_2 V \\
H^2(Y,\sO_Y) \otimes_{\overline \Q} \Omega^2_R @>(1) >> H^2(Y,\sO_Y) \otimes_{\overline \Q}  \Omega^2_\CC\widehat\otimes R
\end{CD}
}
commutes. By Lemma~\ref{lem4}, the image of $\overline{{\rm ch}}_2$ contains all elements
of the form $\rho \otimes \sum_{n=1}^\infty t^n \tau_n$ where $\tau_n = \sum_{i=1}^{p(n)}
da_i^{(n)}\wedge db_i^{(n)}$. Here all the $a_i^{(n)}$ and $b_i^{(n)}$ are chosen
algebraically independent, and we choose a sequence $\{p(n)\}$ such that $\limsup_n
\frac{p(n)}{n}^2 = \infty$. It follows from Lemma~\ref{lem3c} that $\rho \otimes (
\sum_{n=1}^\infty t^n \tau_n)$ does not lie in the image of $(1)$ if $\rho\ne 0$, so the
map labeled $(3)$ cannot be surjective in this case. Note that by Lemma~\ref{appb.gen} a
generic choice of the point $z\in Y(\overline \Q)$ gives rise to non-vanishing $\rho$.
\end{proof}

\begin{rmk}
Of course there are also odd-dimensional varieties $X$, for which algebraization fails. Take
for example  $X=Y\times_{\overline \Q} \P^1$ with $Y$ as in Proposition~\ref{prop:ex}. In
fact any smooth projective $X/\overline \Q$ which maps surjectively onto such a $Y$
does not satisfy algebraization in the sense of  Proposition~\ref{prop:ex}.
\end{rmk}

\end{appendix}

\bigskip

\bibliographystyle{plain}

\renewcommand\refname{References}

\end{document}